\documentclass[12pt]{amsart} 
\usepackage{amscd}
\usepackage{amssymb}
\usepackage{a4wide}
\usepackage{amstext}
\usepackage{amsthm}
\usepackage{xcolor}
\usepackage{xcolor}
\numberwithin{equation}{section}

\newcommand{\Z}{\mathbb{Z}}

\newcommand{\N}{\mathbb{N}}
\newcommand{\R}{\mathbb{R}}

\newcommand{\Cm}{\mathbb{C}}
\newcommand{\Hi}{\mathcal{H}}

\newcommand{\summ}{\sum\limits}
\newcommand{\pho}{\rho}

\newcommand{\eps}{\varepsilon}

 \DeclareMathOperator{\dist}{dist}

\renewcommand{\phi}{\varphi}

\newcommand{\vep}{\varepsilon}
\newcommand{\co}{\mathbb{C}}

\newcommand{\ZZ}{\mathcal{Z}}

\newtheorem{Thm}{Theorem}[section]
\newtheorem{theorem}[Thm]{Theorem}
\newtheorem*{example}{Example}
\newtheorem{lemma}[Thm]{Lemma}
\newtheorem{proposition}[Thm]{Proposition}

\newtheorem{remark}[Thm]{Remark}

\newtheorem*{question}{Question}

\begin{document}
\sloppy
\title[Spectral synthesis for exponentials and logarithmic length]
{Spectral synthesis for exponentials \\ and logarithmic length}
\author{Anton Baranov}
\address{Anton Baranov,
\newline St.~Petersburg State University, St. Petersburg, Russia,
\newline {\tt anton.d.baranov@gmail.com} }
\author{Yurii Belov}
\address{Yurii Belov,
\newline Department of Mathematics and Computer Science, St.~Petersburg State University, St. Petersburg, Russia,
\newline {\tt j\_b\_juri\_belov@mail.ru} }
\author{Aleksei Kulikov}
\address{Aleksei Kulikov,
\newline Department of Mathematical Sciences, Norwegian University of Science and Technology, NO-7491 Trondheim, Norway and
\newline St.~Petersburg State University, St. Petersburg, Russia,
\newline {\tt lyosha.kulikov@mail.ru} }

\thanks{ The results of Sections 2 and 3 were obtained with the support of Russian Science Foundation grant 19-71-30002.
The results of Sections 4 and 5 were obtained with the support of Russian Foundation for Basic Research
grant 20-51-14001-ANF-a.}

\begin{abstract} {We study hereditary completeness  of systems of exponentials on an interval such
that the corresponding generating function $G$ is small outside of a lacunary sequence 
of intervals $I_k$. We show that, under some technical conditions, 
an exponential system is hereditarily complete if and only if the logarithmic length of the union of these intervals 
 is infinite, i.e., $\sum_k\int_{I_k} \frac{dx}{1+|x|}=\infty$. }
\end{abstract}

\maketitle

\section{Introduction}

Let $\{v_n\}_{n\in \N}$ be a complete and minimal system of vectors in a separable Hilbert 
space $\Hi$, that is, $\overline {\rm Span} \{v_n\} = \Hi$ and 
$\overline {\rm Span} \{v_n\}_{n\ne m}\ne \Hi$ for any $m$. 
For any such sequence there exists its unique biorthogonal system $\{w_n\}_{n\in \N}$ such 
that $(v_n, w_m) = \delta_{nm}$. In general, the system $\{w_n\}_{n\in \N}$ needs not be complete 
(e.g., consider $v_n = e_1 + e_{n+1}$ in $\Hi= \ell^2(\N)$), but even if it is complete,  
it is possible that for some partition $\N = A \cup B$, $A\cap B=\emptyset$, the ``mixed''
system $\{v_n\}_{n\in A} \cup \{ w_n\}_{n\in B}$ is incomplete. If it is not the case 
for any partition $\N = A\cup B$, then we call the system $\{v_n\}_{n\in \N}$ {\it hereditarily complete}. 
Hereditary completeness can be understood as a weakest form of reconstruction 
of a vector $f$ from its generalized Fourier series $$\summ_{n\in\mathbb{N}} (f, w_n) v_n,$$ since it is equivalent 
to the fact that each vector $f\in \Hi$ can be approximated by linear combinations of the partial sums 
of its Fourier series. Clearly, if the Fourier series with respect to the biorthogonal pair $(v_n, w_n)$
admit a linear summation method, then the system $\{v_n\}$ is hereditarily complete.

Hereditarily complete systems are also known as strong $M$-bases or systems which admit {\it spectral synthesis} 
due to the relation of this property to the structure of invariant subspaces for certain
classes of linear operators discovered by A.~Markus \cite{markus}. 
Various geometrical aspects of abstract hereditary complete systems 
were considered in \cite{nik1, nik2} while in \cite{arg, kat} some interesting relations
with operator algebras can be found. 

\subsection{Exponential systems} We are interested in the case when $\Hi = L^2(-\pi, \pi)$ and $v_n$ is a system of exponentials, 
$v_n = e^{i\lambda_n t}$ for some set $\Lambda = \{ \lambda_n\}_{n\in \N}\subset \Cm$. 
R.~Young \cite{young} proved that 
in this case the biorthogonal system is always complete and there has been a number of papers 
establishing hereditary completeness and existence of a linear summation method
for nonharmonic Fourier series under some additional hypothesis about the set $\Lambda$
(see, e.g., \cite{bl} or \cite{sedl, sedl1}). Nevertheless, in \cite{bbb1} an example was constructed 
which shows that in general hereditary completeness for exponential systems { does not necessarily hold}. 
This result was extended to other functional systems (reproducing kernels in de Branges 
spaces of entire functions, Gaussian Gabor systems) in \cite{bbb2, bbb3} 
(see also a survey paper \cite{bb-proc}). It should be mentioned that the synthesis 
for exponential systems fails with one-dimensional defect only:  
each mixed system has codimension at most one \cite{bbb1}. 

The construction in \cite{bbb1} was ingenious, but it had very few ``degrees of freedom'',
i.e. free parameters. Therefore, the structure of such examples remained rather mysterious. 
Our aim is to give a larger class of examples. Moreover, under some regularity conditions
we are able to arrive to a certain qualitative characterization ({\it finite logarithmic length}) 
which we believe is intrinsic for the phenomenon of nonhereditary completeness of 
exponential systems. 

It is well-known that if $\{e^{i\lambda t}\}_{\lambda\in\Lambda}$ is a complete and minimal system, then
the following canonical product converges in the sense of principal value, see, e.g., \cite[Lecture 18, Theorem 4]{lev},
$$G(z)=p.v.\prod_{\lambda\in\Lambda}\biggl{(}1-\frac{z}{\lambda}\biggr{)}=\lim_{R\rightarrow\infty}\prod_{\lambda\in\Lambda, |\lambda|<R}\biggl{(}1-\frac{z}{\lambda}\biggr{)}.$$
The function $G$ is called {\it the generating function} of the system $\{e^{i\lambda t}\}_{\lambda\in\Lambda}$. Numerous properties of exponential systems can be expressed in terms of $G$, see, e.g., \cite{bl, pav}.

\subsection{Logarithmic length} We are interested in the case when the function $G$ is small outside some {\it lacunary} 
sequence of intervals 
\begin{equation}
\label{inter}
I_k = [\pho_k - d_k, \pho_k + d_k], \qquad
2\pho_k\le\pho_{k+1}, \quad d_k \le 0.1\pho_k.
\end{equation}
We prove that under some additional restrictions the system 
of exponentials (reproducing kernels) is hereditarily complete 
if and only if the total logarithmic length of these intervals is infinite, that is,
$$
\sum_{k = 1}^\infty \frac{d_k}{\pho_k} = \infty,
$$ 
see Theorems \ref{main1}, \ref{main2}.

{ To illustrate this we present one example which immediately follows from Theorems \ref{main1} and \ref{main2}.
\begin{example} Let $\{\lambda_n\}=\Lambda$ be a localy dense real sequence, i.e., $\sup_n|\lambda_{n+1}-\lambda_n|<\infty$ such that the
generating function $G$ is of exponential type $\pi$, and 
\begin{equation}
\frac{|G(x)|}{\dist(x,\Lambda)}\asymp \max_k{\frac{1}{\sqrt{I_k}(\dist^2(x,I_k)+1)}},
\label{gest}
\end{equation}
where $\{I_k\}$ is a lacunary system of intervals satisfying \eqref{inter} and $d_k\slash \rho_k < 1\slash k$. Then the system $\{e^{i\lambda t}\}_{\lambda\in\Lambda}$ is hereditarily complete if and only if $\sum_{k = 1}^\infty \frac{d_k}{\pho_k} = \infty$.
\end{example}

The existence of sequences $\Lambda$ satisfying \eqref{gest} can be deduced via standard atomization technique, see, e.g., \cite{bel}. From \cite[Lecture 18, Theorem 4]{lev} it follows that the system  $\{e^{i\lambda t}\}_{\lambda\in\Lambda}$ is always complete and minimal in $L^2(-\pi,\pi)$.
}

In Section \ref{exa} we apply these results to give an example of a nonhereditarily complete
(i.e., complete and minimal but not hereditarily complete) 
exponential system which partially answers a problem posed in \cite{bb-proc}: which perturbations of integers 
{can produce} complete and minimal systems of exponentials which are not hereditarily complete?
Let $\lambda_n \in \mathbb{R}$ and
\begin{equation}
\label{kad}
\delta = \sup_{n\in\mathbb{Z}} |\lambda_n - n|.
\end{equation}
By the results of Kadets and Ingham any sequence with $\delta<1/4$  generates a Riesz basis
of exponentials (see, e.g., \cite[Part D, Chapter 4]{nk}). One can ask, however, for which $\delta$ 
any complete and minimal system $\{e^{i\lambda_n t}\}$ satisfying \eqref{kad} \
is automatically hereditarily complete. 

{ \begin{question} Find $\delta_{crit}$ which is the infimum of $\delta>0$ such that there exists nonhereditarily complete system $\{e^{i\lambda_n t}\}$
with $|\lambda_n-n|<\delta$.
\end{question}

The exact value of the {\it synthesis constant} $\delta_{crit}$ is not known. Theorem \ref{exam} shows that
such $\delta_{crit}$ cannot exceed $1/2$. Therefore,
$$\frac{1}{4}\leq \delta_{crit}\leq \frac{1}{2}.$$ 
}

\subsection{Paley-Wiener space\label{pwss}} The classical approach to the study of the properties of exponential systems is to 
consider Fourier transform $\mathcal{F}$ of our system: in this case the Hilbert space becomes 
the Paley--Wiener space $PW_{\pi} = \mathcal{F} L^2(-\pi, \pi)$ 
(the space of all entire functions of exponential type at most $\pi$ which belong to $L^2(\mathbb{R})$)
and the functions $e^{i\lambda_n t}$ are mapped to the reproducing kernels (cardinal sines)
$$
K_{\lambda_n}(z) = \frac{\sin \pi(z-\lambda_n)}{\pi(z-\lambda_n)}
$$
corresponding  to the points $\bar \lambda_n$. In the case when the exponential system is complete and minimal, 
its biorthogonal system $\{w_n\}$ is mapped to the functions 
$G_n\in PW_\pi$ which vanish on $\{\lambda_m, m\ne n\}$. It is 
easy to see that 
$$
G_n(z) = \frac{G(z)}{G'(\lambda_n)(z-\lambda_n)}
$$ 
{ where $G$ is {the generating function} of the system 
$\{e^{i\lambda_n t}\}$ (or, simply put, of the set $\Lambda$).}
The function $G$ vanishes on $\Lambda$ and has no other zeros, it is of exponential type $\pi$ 
(with the diagram $[-\pi i, \pi i]$). Clearly, $G\notin PW_\pi$, however 
$G\in L^2(\R, \frac{dx}{1+x^2})$. Thus, the spectral synthesis problem for exponentials
is equivalent to the same  problem for systems of reproducing kernels
in  the Paley--Wiener spaces. { This equivalence will be frequently used.}

\subsection*{Organization of the paper}In Section \ref{s1} we give a sufficient condition
for an exponential system to be nonhereditarily complete (Theorem \ref{main1}, 
the case of finite logarithmic length). 
In Section \ref{exa} we apply this result to give an explicit example 
of a nonhereditarily complete system of exponentials whose frequencies are 
sufficiently small perturbations of integers. In Section \ref{s3} we prove a converse result
(Theorem \ref{main2})
establishing hereditary completeness in the case of infinite logarithmic length. 
Finally, in Section \ref{s4}, we show that for an incomplete mixed system its exponential part 
must be always sufficiently irregular and, in particular, cannot be { a part of }  Riesz basis of
exponentials with some additional regularity. 
\bigskip


\section{Case of finite logarithmic length}
\label{s1}



In what follows we will need a number of assumptions on $\Lambda$ and $G$:
\begin{enumerate}
\item[(a)] $\Lambda \subset \R$ and $G(\R) \subset \R$; 
\item[(b)] $\dist (\Z, \Lambda) > 0$; 
\item[(c)] $\Lambda$ is locally dense, i.e., 
there exists some $C > 0$ such that any interval $I\subset \R$, $|I| \ge C$, contains 
at least one element of $\Lambda$; 
\item[(d)] $|G(iy)| = o(e^{\pi |y|}), |y|\to \infty$. 
\end{enumerate}

Note that the series $\sum_{n\in\mathbb{Z}}|G(n)|^2$ diverges for any generating function of a complete and minimal system satisfying condition (d). Otherwise it is easy to show that $G$ belongs to the Paley--Wiener space which contradicts completeness.

For our scheme to work, on the most part of $\Z \backslash \bigcup I_k$ we just need
$G(n)$ to be $\ell^2$-sequence. But in some neighborhood of $I_k$ we want $G(n)$ to be slightly 
better than $\ell^2$ (see condition (ii) below). 

\begin{theorem}
\label{main1}            
Let $G$ be the generating function of some complete and minimal system 
$\{e^{i\lambda t}\}_{\lambda\in \Lambda}$ satisfying \textup(a\textup)--\textup(d\textup).
Let $I_k$ be a system of intervals of the form \eqref{inter} such that
$\dist(\pho_k, \Z)\ge \frac{1}{3}$ for each $k$ and $G(\pho_k) \ne 0$.
Put $g_k = \summ_{n\in I_k} G^2(n)$, $s_k = \sqrt{d_kg_k\pho_k}$ and 
$$
J_k = J_k^- \cup J_k^+ = [\pho_k -d_k- 2s_k, \pho_k - d_k-s_k]\cup [ \pho_k + d_k + s_k, \pho_k + d_k + 2s_k].
$$ 
Assume that the function $G$ satisfies the following conditions\textup:

\begin{enumerate}
\item $\{ G(n): \  n\in \Z \backslash \bigcup\limits_k I_k\} \in \ell^2$\textup;

\item $\summ_k s_k\summ_{n\in J_k} G^2(n) < \infty$\textup;

\item $s_k \le 0.1 \pho_k$\textup;

\item $\summ_k \dfrac{d_k}{\pho_k} < \infty$.
\end{enumerate}
Then the system $\{e^{i\lambda t}\}_{\lambda\in \Lambda}$ is not hereditarily complete.
\end{theorem}

\begin{proof}
First, we note that to prove that the system $\{v_n\} = \{e^{i\lambda_n t}\}$  
is not hereditarily complete
it is enough for some splitting $\N = A\cup B$ to construct two vectors $f, g\in \Hi$ 
such that $g\perp \{w_n\}_{n\in B}$, $f\perp \{v_n\}_{n\in A}$ but $f$ and $g$ are not 
orthogonal. Indeed, if the system $\{v_n\}_{n\in A} \cup \{ w_n\}_{n\in B}$ were complete, 
then the vector $g$ would lie in the span of $\{ v_n\}_{n\in A}$ and $f$ would lie in the 
span of $\{ w_n\}_{n\in B}$ and so they would be orthogonal.
\medskip
\\
{\bf Step 1. Construction of $f$.} We will construct a function $f$ as a small perturbation 
of the function $G$ so that 
they will share most of their zeros and therefore $f$ is orthogonal to the most 
of $K_{\lambda_n}$, and then construct function $g$ by the fixed point theorem so that
it is orthogonal to the remaining functions $G_n$.

Without loss of generality we  can throw away those $k$ for which $d_k \ge 0.1 s_k$, that 
is, $g_k\le 0.01 \frac{d_k}{\pho_k}$. Note that since the series $\summ_{k = 1}^\infty \frac{d_k}{\pho_k}$ 
is convergent, all our assumptions are still satisfied. For the remaining 
$k$ it is easy to see that all intervals $I_k, J_k$ are pairwise disjoint. Also later we 
will throw away some extra finite set of intervals -- it also will not break any of our assumptions.

Let $t_k\in J_k$ be a zero of $G$ whose choice will be specified later. Put 
$$
f(z) = G(z)m(z),
$$ 
where 
$$
m(z) = \prod\limits_{k = 1}^\infty \frac{1 - z/\pho_k}{1 - z/t_k}. 
$$
It is easy to see from the lacunarity of $\pho_k$ that this product converges locally 
uniformly on $\Cm \backslash \{t_k\}$, and since $G(t_k) = 0$ we conclude that the function $f$ is entire.
\medskip
\\
{\bf Step 2. $f \in PW_\pi$.} 
We will select two candidates $t_k^{\pm} \in \frac{1}{2}J_k^{\pm}$ for $t_k$ (as always, 
by the half of the interval we mean the interval with the same center and twice smaller 
length). Then we choose one of them in such a way that $0.001\le\prod\limits_{k = 1}^N \
\frac{t_k}{\pho_k}\le 1000$ for all $N$ (we can always do that by (iii)). Note that in this case  
we have 
\begin{equation}
\label{rain}
|m(z)| \asymp \bigg| \frac{z-\pho_k}{z-t_k} \bigg|, \qquad
\frac{\pho_{k-1} + \pho_k}{2} \le |z| \le \frac{\pho_k + \pho_{k+1}}{2}.
\end{equation}
Therefore the function $f$ is 
of exponential type at most $\pi$ and $|f(iy)| = o(e^{\pi |y|})$. Thus, to 
prove that $f\in PW_\pi$, it is enough to show that $\{f(n)\} \in \ell^2(\Z)$.

Trivially, $\{f(n)\} \in \ell^2(\Z \backslash \bigcup\limits_k (I_k \cup J_k))$ since 
$|m(n)|\lesssim 1$ for those $n$. For $n\in I_k$ we have $\frac{|x-\pho_k|}{|x-t_k|} 
\le \frac{d_k}{s_k}$ and, therefore, 
\begin{equation*}
\summ_{k = 1}^\infty \summ\limits_{n\in I_k} |f(n)|^2 \lesssim 
\summ_{k = 1}^\infty \frac{d_k^2}{s_k^2}g_k = \summ_{k = 1}^\infty \frac{d_k}{\pho_k} < \infty.
\end{equation*}

For $n\in J_k$ we have $|m(n)| \asymp \frac{s_k}{n - t_k}$. Let us divide 
$\frac{1}{2}J_k^{\pm}$ into $r_k\asymp s_k$ intervals of length $C$ and 
choose one root of $G$ from each of them (there is always at least one by 
our assumption (iii)). Denote these roots in $J_k^+$ by $\lambda_j$, $j=1, \dots l$.
Then 
$$
\sum_{j=1}^l \sum_{n\in J_k} \frac{s_k^2 G^2(n)}{(n-\lambda_j)^2}  \lesssim
s_k^2 \sum_{n\in J_k} G^2(n),
$$
whence there exists $\lambda_j$ such that 
$$
\sum_{n\in J_k} |f(n)|^2 \lesssim   
\sum_{n\in J_k} \frac{s_k^2 G^2(n)}{(n-\lambda_j)^2} \lesssim s_k  \sum_{n\in J_k} G^2(n)
$$
with constants in $\lesssim$ independent on $k$. Similarly, one can choose $t_k^- \in J_k^-$. 
By condition (ii), $\summ_k \summ_{n\in J_k} |f(n)|^2 <\infty$ both for
$t_k = t_k^+$ and $t_k = t_k^-$.
\medskip
\\
{\bf Step 3. Construction of the function $g$.}
Put $a_n = (-1)^nf(n)$. We are going to construct a real sequence  $\{ b_n\}\in \ell^2(\Z)$ such 
that $\summ_{n\in \Z}a_nb_n\ne 0$ and the function $S(z) = \summ_{n\in \Z} \frac{a_nb_n}{z-n}$ 
has zeros at each $\pho_k$. 

Let us show that once such system $\{b_n\}$ is constructed, 
the functions $f(x)$ and 
$$
g(z) = \frac{\sin \pi z}{\pi} \summ_{n\in \Z} \frac{b_n}{z-n} = 
\summ_{n\in \Z} (-1)^n b_n K_n(z)
$$ 
will achieve our goals. By construction, $f$ is orthogonal to all 
$K_{\lambda_n}$ except for $t_k$ and $(f, g) = \summ_{n\in \Z} a_nb_n \ne 0$. 
It remains to prove that $g$ is orthogonal to $\frac{G(z)}{z-t_k}$ for all $k$
whence the mixed system 
$$
\{K_{\lambda}\}_{\lambda\in \Lambda_1} \cup \big\{\frac{G(z)}{z-\lambda} \big\}_{\lambda\in \Lambda_2}
$$
with $\Lambda_2 = \{t_k\}_{k\ge 1}$, $\Lambda_1 = \Lambda \setminus\Lambda_2$,
is incomplete. Thus, we need to prove that
\begin{equation}
\label{G g orth}
\Big(\frac{G(z)}{z-t_k}, g\Big)= \summ_{n\in\Z} \frac{(-1)^n b_n G(n)}{n - t_k}=0.
\end{equation}

We are going to prove that 
$$
\frac{G(z)S(z)}{f(z)} = \summ_{n\in \Z} \frac{(-1)^n b_n G(n)}{z-n}.
$$ 
If we do so, then substituting $z = t_k$ we get \eqref{G g orth} (note that $f(t_k) \ne 0$).
Consider the function 
$$
H(z) = \frac{G(z)S(z)}{f(z)} - \summ_{n\in \Z} \frac{(-1)^n b_nG(n)}{z-n}.
$$ 
Trivial computation shows that its residues at $\Z$ are zero and since $GS$ vanish in all zeros of $f$
we conclude that $H$ is entire. On the other hand, by comparing indicator 
diagrams of corresponding functions we see that $H$ is of minimal exponential type. 
Finally, from $|m(iy)|\asymp 1, y\to \infty$ and the definitions of $f$ and $S$ we 
see that $|H(iy)| = o(1), |y|\to \infty$. Therefore, by Phragm\'en--Lindel\"of principle, $H \equiv 0$.

It remains to construct a sequence $\{b_n\}_{n\in \Z}$ with the desired properties.
\medskip
\\
{\bf Step 4. Construction of the sequence $\{b_n\}$.}
Put $b_0 = \frac{1}{a_0}$, $b_n = 0$ for $n\in \mathbb \Z \backslash \bigcup I_k, n\ne 0$,
and $b_n = c_k \frac{a_n}{\pho_n - n}$ for $n\in I_k$ for some $c_k$. We want to construct 
a sequence $c_k$ such that $S(\pho_k) = 0$ for all $k$.

Consider the Banach space $\mathcal{B}$ of sequences $\{c_k\}_{k\ge 1}$ 
with the norm $\|c\|_{\mathcal{B}} = \sup\limits_{k\ge 1} \frac{|c_k|}{d_k}$, 
and consider the following operator on it:
$$
\begin{aligned}
(Tc)_k & = \bigg(\summ_{n\in I_k} \frac{a_k^2}{(\pho_k - n)^2}\bigg)^{-1} 
\bigg( -\frac{1}{\pho_k} - \summ_{m\ne k}\summ_{n\in I_m} \frac{a_nb_n}{\pho_k - n}\bigg) \\
& = 
\bigg(\summ_{n\in I_k} \frac{a_k^2}{(\pho_k - n)^2}\bigg)^{-1} 
\bigg( -\frac{1}{\pho_k} - \summ_{m\ne k}c_m \summ_{n\in I_m} \frac{a_n^2}{(\pho_k - n)(\pho_m - n)}\bigg).
\end{aligned}
$$
It is easy to see that if $\{c_k\}$ is a fixed point of this operator then $S(\pho_k) = 0$ 
for all $k$. Thus it remains to prove that $T$ is contractive.

Recall that $(-1)^n a_n = f(n) = G(n)/m(n)$. We have
\begin{equation}
\label{denom}
\summ_{n\in I_k} \frac{a_k^2}{(\pho_k - n)^2} = 
\summ_{n\in I_k} \frac{G^2(n)}{m^2(n)(\pho_k -n)^2} \asymp \summ_{n\in I_k} 
\frac{G^2(n)}{(t_k - n)^2} \asymp \frac{g_k}{s_k^2} \asymp \frac{1}{d_k\pho_k}.
\end{equation}
On the other hand, by \eqref{rain},
\begin{equation*}
\summ_{n\in I_m} \frac{a_n^2}{|(\pho_k - n)(\pho_m - n)|} 
\lesssim \frac{1}{\pho_k} \summ_{n\in I_m} \frac{a_n^2}{|\pho_m - n|} 
\lesssim \frac{1}{\pho_k} \summ_{n\in I_m} \frac{G^2(n)|\pho_m - n|}{|t_m - n|^2} \le \frac{g_md_m}{\pho_k s_m^2} = \frac{1}{\pho_k \pho_m}.
\end{equation*}
Clearly, the operator $T$ is the sum of a constant vector and some linear operator. 
Moreover, by \eqref{denom}, this vector is in $\mathcal{B}$. So it remains to prove that the linear part 
of $T$ is contractive. Denoting it by $T_{lin}$ we have
\begin{equation*}
\frac{|(T_{lin} c)_k|}{d_k} \lesssim \frac{d_k \pho_k}{d_k} \summ_{m\ne k} 
\frac{|c_m|}{\pho_k\pho_m} \le \|c\|_{\mathcal{B}} \summ_{m}\frac{d_m}{\pho_m}.
\end{equation*}
As we mentioned in the beginning of the proof, we can safely throw away any finite number 
of intervals. Thus, we can assume that $\summ_m \frac{d_m}{\pho_m}$ is as small as 
we like and so $\|T_{lin}(c)\|_{\mathcal{B}} 
\le \frac{\|c\|_{\mathcal{B}}}{2}$. Therefore, $T$ has a fixed point $c$.

It remains to prove that $b_n\in \ell^2(\Z)$ and $\summ_{n\in \Z} a_nb_n \ne 0$. We have,
by \eqref{denom},
\begin{equation*}
\summ_{n\in \Z} |b_n|^2 = |b_0|^2 + \summ_k |c_k|^2 \summ_{n\in I_k} 
\frac{a_n^2}{(\pho_k - n)^2} \lesssim |b_0|^2 + 
\|c\|_{\mathcal{B}} \summ_k \frac{d_k}{\pho_k} < \infty.
\end{equation*}
and, again using \eqref{denom},
\begin{equation*}
\summ_n a_nb_n = 1 + \summ_k c_k \summ_{n\in I_k} \frac{a_n^2}{\pho_k - n} \ge 1 - \|c\|^2_{\mathcal{B}}
\summ_k d_k \summ_{n\in I_k} \frac{a_n^2}{|\pho_k - n|} \ge 1 - A \|c\|^2_{\mathcal{B}} 
\summ_k \frac{d_k}{\pho_k}
\end{equation*}
for some constant $A$ depending only on $\dist (\Z, \Lambda)$ from (b). 
We can once again throw away some intervals $I_k$ so that the expression in the right-hand side be 
positive (note that since $\|T_{lin}\|\le 1/2$, 
we can give uniform upper bound for $\|c\|_{\mathcal{B}}$ so it is enough to make 
$\summ_k \frac{d_k}{\pho_k}$ sufficiently small).
\end{proof}

\begin{remark}
{\rm We can replace condition (iii) by $s_k \le C\pho_k$ -- just replace $s_k$ with 
$\eps s_k$ for sufficiently small $\eps$.}
\end{remark}
{ 
\begin{remark}
{\rm Note that in the proof we actually need only that a locally dense subset of the zeros of $G$ has a positive distance from $\Z$.}
\end{remark}

We will now use the Theorem \ref{main1} to construct a completely explicit example of a function $G$ which gives us a nonhereditarily complete system.

\begin{theorem}
Let $\Lambda$ be the set of zeros of the function
\begin{equation}\label{simple example}
G(x) = \cos \pi x \left (\frac{1}{x-1/2} +  \summ_{k = 10}^\infty \bigg(\frac{1}{x-2^k+1/2} - \frac{1}{x-2^k-1/2}\bigg)\right).
\end{equation}
Then the system $\{ e^{i\lambda t}\}_{\lambda\in \Lambda}$ is not hereditarily complete.
\end{theorem}
\begin{proof}

It is easy to see that $G \notin PW_ \pi$, but $G \in PW_\pi + zPW_\pi$. Moreover, $|G(z)| \gtrsim |z|^{-1}e^{\pi |{\rm Im} z|}$ for $|z| = 2^k + 2^{k -1}$ and $k \in \mathbb{N}$ is sufficiently large. Therefore, one cannot multiply $G$ by an entire function and remain in $PW_\pi$. By \cite[Lecture 18, Theorem 4]{lev} $G$ is the generating function of some complete and minimal system of exponentials.

Put $\pho_k = 2^k$ and $d_k = \pho_k^{1/5}$. Conditions $(a)-(d)$ and $(iii)$, $(iv)$ are easy to verify 
(for conditions $(b), (c)$ see the remark above). Let us now verify conditions $(i)$ and $(ii)$. We begin with the condition $(i)$.

For $|n|\in [\frac{\pho_{k-1} + \pho_k}{2}, \frac{\pho_k + \pho_{k+1}}{2}]$ we have from \eqref{simple example}
\begin{equation}\label{boundex}
|G(n)| \lesssim  \frac{1}{|n-\pho_k|^2} + \frac{1}{2^k} + \frac{k}{4^k} \lesssim \frac{1}{|n-\pho_k|^2} + \frac{k}{2^k}.
\end{equation}
Therefore 
\begin{equation*}
\sum_{n\notin \cup I_k} G^2(n) \lesssim \sum_{k=1}^\infty \bigg( \frac{1}{d_k^3} + \frac{k^2}{2^k}\bigg) < \infty.
\end{equation*}

To prove $(ii)$ note that $g_k =  \summ_{n\in I_k} G^2(n) \asymp 1$ and so $s_k = \pho_k^{3/5}$. By the bound \eqref{boundex} we get
\begin{equation}
\sum_{k = 1}^\infty s_k \sum_{n\in J_k} G^2(n) \lesssim \sum_{k = 1}^\infty s_k \bigg(\frac{1}{s_k^3} + \frac{k^2s_k}{4^k}\bigg) < \infty.
\end{equation}
\end{proof}

Note that the minus sign in \eqref{simple example} is absolutely essential to verify the condition $(ii)$, similar cancelation can be observed implicitly in the example from  \cite{bbb1}. Moreover, if we enumerate zeros of the function $G$ in increasing order then for all $n\in \Z$  we would have (after shifting by $\frac{1}{2}$) $|\lambda_n - n| \leq 1$ just as in \cite{bbb1} (one can check that all the roots of the function $G$ are real). The drawback of these examples is that they do not use the full potential of the Theorem \ref{main1} -- we could have chosen $d_k$ as any positive power of $\pho_k$ and the analysis would still work. In the following section we will construct a more advanced example which will break this barier and give us $|\lambda_n - n| < \frac{1}{2} + \eps$.
}
\bigskip

\section{Example of a nonhereditarily complete sequence}
\label{exa}

In this section we give an explicit example of a sequence $\Lambda$ satisfying conditions of Theorem
\ref{main1}. Moreover, this system will be a sufficiently small perturbation of integers. 

\begin{theorem}
\label{exam}
For any $\delta > \frac{1}{2}$ there exist $G$ and $\Lambda$ satisfying all 
conditions of Theorem \ref{main1} and such that $\Lambda$ satisfies \eqref{kad}.
\end{theorem}

\begin{proof}                                                      
We will start with the following auxiliary function. Let $\delta_0 \in [1/2, 1)$. Consider
the function 
$$
G_0(z) = (z-1/2) \prod_{n\in\N} \bigg(1-\frac{z^2}{(n+\delta_0)^2}\bigg). 
$$
It is well known that
$$
|G_0(x)| \asymp (|x|+1)^{-2\delta_0} {\rm dist} (x, \mathcal{Z}_{G_0})
$$
(here and in what follows we denote by $\mathcal{Z}_F$ the zero set of an entire function $F$), whence 
$G_0 \in PW_\pi$ and, in particular, $|G_0(iy)| = o(e^{\pi |y|}), |y|\to \infty$. 

Now let $\delta\in (\delta_0, 1)$. The idea is to shift a part of the zeros of $G_0$ which belong
to some lacunary sequence of intervals $I'_k = [\rho_k - d_k, \rho_k + d_k]$ back to the origin. 
Let $\rho_k$ be an arbitrary sequence such that $\rho_{k+1}>2\rho_k$ and choose $d_k$ so that
\begin{equation}
\label{dd}
d_k^{\delta_0+\delta} = \rho_k^{2\delta_0}.
\end{equation}
Of course, we may assume that $d_k\le  \rho_k/100$ for all $k$. Now put
$$
G(z) = G_0(z) \prod_k\prod_{n\in I'_k} \frac{z- (n-\delta)}{z-(n+\delta_0)} = 
G_0(z) \prod_k\prod_{n\in I'_k}\bigg(1 + \frac{\delta+\delta_0}{z-(n+\delta_0)}\bigg).
$$

{
Let $x\in \big[\frac{\rho_{k-1} +\rho_k}{2}, \frac{\rho_{k} +\rho_{k+1}}{2}\big]$. Then
$$
\sum_{m\ne k} \sum_{n\in I'_m} \frac{1}{|x-(n+\delta_0)|} \lesssim \frac{1}{\rho_k} \sum_{m<k} d_m + \sum_{m>k} \frac{d_m}{\rho_m}.
$$
Since $\sum_k \frac{d_k}{\rho_k} <\infty$, the product converges and, moreover, 
$$
\bigg|\prod_{m\ne k} \prod_{n\in I'_m}\bigg(1 + \frac{\delta+\delta_0}{x-(n+\delta_0)} \bigg) \bigg| \asymp
1, \qquad x\in \Big[\frac{\rho_{k-1} +\rho_k}{2}, \frac{\rho_{k} +\rho_{k+1}}{2} \Big].
$$
Also, let $n_0+\delta_0$ and $n_1-\delta$, $n_0, n_1 \in I_k'$, be respectively the zeros of $G_0$ and $G$ closest to $x$.
Then
$$
\begin{aligned}
\log \bigg| \prod_{n\in I'_k}\bigg(1 + \frac{\delta+\delta_0}{x-(n+\delta_0)} \bigg) \bigg| & = 
\log\frac{|x-(n_1-\delta)|}{|x-(n_0+\delta_0)|} + \sum_{n\in I_k', n\ne n_0, n_1} \frac{\delta+\delta_0}{x-(n+\delta_0)} + O(1) \\
& =
\log\frac{|x-(n_1-\delta)|}{|x-(n_0+\delta_0)|} + (\delta+\delta_0) \ln \frac{|x-(\rho_k-d_k)|+1}{|x-(\rho_k+d_k)|+1} + O(1).
\end{aligned}
$$
Thus, for $x\in \big[\frac{\rho_{k-1} +\rho_k}{2}, \frac{\rho_{k} +\rho_{k+1}}{2}\big]$, we have 
\begin{equation}
\label{gg}
|G(x)| \asymp |G_0(x)|\cdot \frac{{\rm dist}(x, \mathcal{Z}_G)}
{{\rm dist}(x, \mathcal{Z}_{G_0})}\cdot
\bigg(\frac{|x-(\rho_k-d_k)|+1}{|x-(\rho_k+d_k)|+1}\bigg)^{\delta_0+\delta}.
\end{equation}  }
We will show that $G$ satisfies all conditions of Theorem \ref{main1}
with $I_k = [\rho_k - 2d_k, \rho_k + 2d_k]$. 

Obviously, $G$ is an entire function of exponential type $\pi$ (with diagram $[-\pi i, \pi i]$)
and $|G(iy)| = o(e^{\pi|y|})$, $y\to \infty$, since $|G(iy)|\asymp |G_0(iy)|$. Thus, all conditions (a)--(d)
are satisfied. 

Let us show that $G$ is the generating function of some complete and minimal system
of reproducing kernels in $PW_\pi$. It is clear from \eqref{gg} that $|G(x)| \gtrsim 
(|x|+1)^{-K} {\rm dist}(x, \mathcal{Z}_G)$, $x\in \mathbb{R}$,  for some $K>0$. 
Thus, if $GT \in PW_\pi$ for some entire function $T$ of zero exponential type, then $T$ 
is a polynomial. However, by \eqref{dd}, for any $n\in [\rho_k+d_k, \rho_k+d_k +2]$ we have
$|G(n)| \asymp \rho_k^{-2\delta_0} d_k^{\delta_0 + \delta} =1$. 
Thus, $GT\notin PW_{\pi}$ for any polynomial $T$.
By \eqref{gg} we also have 
$$
|G(x)| \asymp (|x|+1)^{-2\delta_0} {\rm dist}(x, \mathcal{Z}_G), \qquad 
x\notin \cup I_k, 
$$
and so  $\{G(n): n\in \Z \setminus \cup_k I_k\} \in \ell^2$. Also, 
$$
1\lesssim g_k = \sum_{n\in I_k}|G(n)|^2 \lesssim \frac{d_k^{\delta_0+\delta}}{\rho_k^{2\delta_0}}
\sum_{n\in I_k}\frac{1}{ (|n-(\rho_k+d_k)|+1)^{\delta_0+\delta} } \lesssim 1.
$$
It follows that $\frac{G}{z-\lambda} \in PW_{\pi}$ for any $\lambda\in \mathcal{Z}_G$ and so 
$G$ is the generating function of some complete and minimal system which satisfies (i) of Theorem \ref{main1}.

Since $g_k\asymp 1$ we have $s_k\asymp \sqrt{\rho_k d_k} <\rho_k/100$
for sufficiently large $k$. It remains to verify (ii). 
Let $J_k = [\rho_k-d_k-2s_k, \rho_k - d_k-s_k] \cup [\rho_k+d_k +s_k, \rho_k + d_k + 2s_k]$.
Then
$$
\sum_k s_k \sum_{n\in J_k} G^2(n) \lesssim \sum_k s^2_k \rho_k^{-4\delta_0} = 
\sum_k \rho_k^{1+\frac{2\delta_0}{\delta+\delta_0}-4\delta_0} < 
\sum_k \rho_k^{\frac{2\delta_0}{\delta+\delta_0}-1}.
$$
Since $\delta> \delta_0$, we conclude that the above some converges.    

Note that the constants $\delta>\delta_0\ge 1/2$ were arbitrary and so Theorem \ref{exam} is proved.
It is clear from the last step of the proof that the condition $\delta_0\ge 1/2$ 
is essential for this construction.  
\end{proof}
\bigskip


\section{Case of infinite logarithmic length}
\label{s3}

Throughout this section the symbols $I_k$, $\rho_k$, $d_k$ and $g_k$ will have the same 
meaning as in Section \ref{s1} (note that $J_k$ will denote a different object).

For the converse theorem we need to somehow formalize the statement ``$G$ is big on the intervals $I_k$''. 
This must include the following two ingredients: first of all we do not want, for some 
$k$, the sum $g_k = \summ_{n\in I_k} |G(n)|^2$ to be significantly larger than the same sum for 
its neighbours because otherwise we will not ``feel'' them in $G$ and the total logarithmic 
length may become finite; secondly, we do not want the main contribution to $g_k$
be due to the values of $G$ on a small part of $I_k$ because in that case we will only ``feel'' 
this small part of $I_k$ and logarithmic length may again become finite. {These two parts corresponds to the assumptions $(i)$ and $(ii)$ of the following theorem.}

\begin{theorem}
\label{main2}
Let $G$ be the generating function of some complete and minimal system 
$\{e^{i\lambda t}\}_{\lambda\in \Lambda}$ such that 
$\Lambda \cap \Z =\emptyset$ and 
$$
\summ_{n\in \Z} \frac{|G(n)|^2 +|G(n)|}{|n|+1}<\infty.
$$
Assume that there exists a constant $C > 0$ such that for all $k$ we have
\begin{enumerate}
\item $\summ_{n\notin I_k} \dfrac{|G(n)|^2}{|n - \pho_k|} \le C\dfrac{g_k}{d_k}$\textup;


\item $\sqrt{\dfrac{g_k}{d_k}} \le C|G(x)|$, \quad $x\in I_k$\textup;

\item $\summ_k \dfrac{d_k}{\pho_k} = \infty$.
\end{enumerate}
Then the system $\{ e^{i\lambda t}\}_{\lambda \in \Lambda}$ is hereditarily complete.
\end{theorem}


For the proof of Theorem \ref{main2} we need the following proposition.

\begin{proposition}
\label{lem1}
Let $t_k\in \mathbb{R}$ be an increasing sequence \textup(one-sided or two-sided\textup) 
which is separated, i.e., $t_{k+1} - t_k \ge \delta$ for some $\delta>0$,
and let $\mu_k \ge 0$, $\{ \mu_k\} \in \ell^1(\Z)$. 
Let $\{\gamma_n\}$ be an increasing separated sequence such that 
${\rm dist}(\{\gamma_n\}, \{t_k\}) = d>0$ 
and $\summ_n \frac{1}{\gamma_n} = \infty$. Then, if for the function
$$ 
f(z) = \summ_k \frac{\mu_k}{z-t_k}
$$
we have $\{f(\gamma_n)\} \in \ell^1(\N)$, then $\mu_k \equiv 0$.
\end{proposition}

\begin{proof}[Proof of Proposition \ref{lem1}]
Let us assume that $\sup \{t_k\} = \infty$. Otherwise, $f(x) \asymp x^{-1}$, $x\to\infty$, 
and the statement is trivial. 

Without loss of generality we can assume that all $\mu_k$ are positive. Then it is clear that $f$ has a unique zero $s_k$ in each interval $(t_k, t_{k+1})$. It is known and not difficult to show
(see \cite[Proposition 5.4]{bbb1}) that these zeros in a sense approach the 
``outer'' ends of the intervals $(t_k, t_{k+1})$, namely, 
$$
\sum_{t_k>0} \frac{t_{k+1} - s_k}{s_k}<\infty, \qquad\quad
\sum_{t_k<0} \frac{s_k - t_k}{|s_k|}<\infty.
$$

Put $\tilde f(z) = f(z) - \frac{\mu_0}{2(z-t_0)}$ and denote the zeros of 
$\tilde f$ in $(t_k, t_{k+1})$ by $\tilde s_k$. Then we also have
\begin{equation}
\label{asimp0}
\sum_{t_k>0} \frac{t_{k+1} - \tilde s_k}{\tilde s_k}<\infty.
\end{equation}

Consider first those $\gamma_n$ which belong to $E_k = \cup_{t_k>0} (\tilde s_k, t_{k+1})$. Since $\{\gamma_n\}$
is separated and ${\rm dist}(\{\gamma_n\}, \{t_k\}) =d>0$, 
we conclude that for a fixed $k$ the number of points $\gamma_n \in (\tilde s_k, t_{k+1})$
does not exceed $C(t_{k+1} - \tilde s_k)$ for some $C>0$ independent on $k$. In particular, 
the interval $(\tilde s_k, t_{k+1})$ contains no 
points $\gamma_n$ if $t_{k+1} - \tilde s_k <d $. Hence,
$$
\summ_{t_k>0} \summ_{\gamma_n \in E_k} \frac{1}{\gamma_n} 
\lesssim \sum_{t_k>0} \frac{t_{k+1} - \tilde s_k}{\tilde s_k} <\infty. 
$$ 
Thus, we may assume without loss of generality that $\gamma_n \in \cup_{t_k>0} (t_k, \tilde s_k)$
for all $n$. Since $\tilde f$ is decreasing on each interval $(t_k, t_{k+1})$ 
we have $\tilde f (\gamma_n) \ge f(\tilde s_k) = 0$ whenever $\gamma_n \in (t_k, t_{k+1})$, 
and so 
$$
f(\gamma_n) = \frac{\mu_0}{2(\gamma_n-t_0)} + \tilde f (\gamma_n) \gtrsim \frac{1}{\gamma_n}.
$$
This contradicts the assumption that $\{f(\gamma_n)\} \in \ell^1$. 
\end{proof}

In the proof of Theorem \ref{main2} we will need some auxiliary Hilbert space of meromorphic functions
in $\mathbb{C} \setminus\Z$. Put 
$$
\Hi = \bigg\{ \summ_{n\in \Z} \frac{b_n|G(n)|}{z-n}, \{b_n\} \in \ell^2\bigg\}, \qquad
\bigg\langle \summ_{n\in \Z} \frac{b_n|G(n)|}{z-n},  \summ_{n\in \Z} \frac{c_n|G(n)|}{z-n}
\bigg\rangle _{\Hi} = \sum_{n\in \Z} b_n\overline{c_n}.
$$ 
It is easy to see that $\Hi$ is a Hilbert space whose reproducing kernel 
at $\lambda \in \mathbb{C} \setminus\Z$ is given by 
$\mathcal{K}_\lambda(z) = \summ_{n\in \Z} \frac{|G(n)|^2}{(z-n)(\bar \lambda - n)}$ and,
in particular, 
\begin{equation}
\label{rep}
\|\mathcal{K}_\lambda\|_{\Hi}^2 = \sum_{n\in\Z}
\frac{|G(n)|^2}{|\lambda-n|^2}.
\end{equation}

Next, consider the function
\begin{equation}
M(t) = \summ_n \frac{|G(n)|^2}{n-t}.
\end{equation}
Each interval $(n, n+1)$ contains exactly one root of the equation $M(t) = 0$.
Pick those roots which lie in $J_k =\frac{1}{2}I_k = [\pho_k - \frac{d_k}{2}, 
\pho_k + \frac{d_k}{2}]$ for some $k$ and denote the resulting sequence by $\{t_n\}$.

Note that, for real $z\ne w$, we have
$$
\langle  \mathcal{K}_w, \mathcal{K}_z \rangle = \mathcal{K}_w(z) = \frac{M(z) - M(w)}{z-w}.
$$
Since $M(t_n) = 0$ for all $n$, the functions $\mathcal{K}_{t_n}$ form an orthogonal system 
in $\Hi$. 

\begin{lemma}
\label{lemm2}
There exist sets $N_k \subset \mathbb{N} \cap J_k$ and a constant $C>0$ \textup(independent on $k$\textup) 
such that $|N_k|\ge d_k/2$ and 
\begin{equation}
\label{ddd}
\|\mathcal{K}_{n+\frac{1}{2}}\|^2_{\Hi} \le C \frac{g_k}{d_k}, \qquad
\Big| M\Big(n+\frac{1}{2}\Big)\Big| \le C \frac{g_k}{d_k}, \qquad n\in N_k.
\end{equation}
\end{lemma}

\begin{proof}
Throughout the proof, symbol $C$ will denote different constants independent on $k$.
We have 
$$
\summ_{n\in J_k} \|\mathcal{K}_{n+\frac{1}{2}}\|_{\Hi}^2 = 
\summ_{n\in J_k} \summ_{m\in\Z} \frac{|G(m)|^2}{(n-m+\frac{1}{2})^2}. 
$$
By condition (i) of Theorem \ref{main2}, 
$$
\summ_{n\in J_k} \summ_{m\notin I_k} \frac{|G(m)|^2}{(n-m+\frac{1}{2})^2} \le C g_k,  
$$
while 
$$
\summ_{n\in J_k} \summ_{m\in I_k} \frac{|G(m)|^2}{(n-m+\frac{1}{2})^2} =
\summ_{m\in I_k} |G(m)|^2 \summ_{n\in J_k} \frac{1}{(n-m+\frac{1}{2})^2} \le Cg_k.
$$
We conclude that 
$\summ_{n\in J_k} \|\mathcal{K}_{n+\frac{1}{2}}\|_{\Hi}^2 \le Cg_k$, whence,
for any $\vep>0$, we have $\|\mathcal{K}_{n+\frac{1}{2}}\|_{\Hi}^2 \le 2C\vep^{-1} \frac{g_k}{d_k}$
for some set $N_k$ with $|N_k|\ge (1-\vep)d_k$. 

{
Estimate for $|M(n+\frac{1}{2})|$ is more delicate. First, we split it into the sums over $I_k$ and $\Z \backslash I_k$
\begin{equation}
M\biggl{(}n+\frac{1}{2}\biggr{)} = \sum_{m\in I_k} \frac{|G(m)|^2}{m-n-\frac{1}{2}} + \sum_{m\notin I_k} \frac{|G(m)|^2}{m-n-\frac{1}{2}} = S_1(n) + S_2(n).
\end{equation}
For $n\in J_k$ as above we can deduce from the  assumption (i) that $|S_2(n)| \le C\frac{g_k}{d_k}$. Therefore it remains to estimate $S_1(n)$.

For a sequence $x = (x_m)\in \ell^1(\Z)$ consider the operator $T$ defined as
\begin{equation}
(Tx)_n = \summ_{m\in \Z} \frac{x_m}{m-n-\frac{1}{2}}.
\end{equation}
$T$ is a discrete Hilbert transform and as such it has a weak-type $(1, 1)$ bound
\begin{equation}
|\{ n\in \Z : |(Tx)_n| > \lambda\}| \le C\frac{||x||_{\ell^1}}{\lambda},
\end{equation}
where $C$ is an absolute constant.

Applying this bound to the sequence $x_n = |G(n)|^2 \chi_{I_k}(n)$ with $\lambda = 100Cg_k\slash d_k$ we get
\begin{equation}
|\{ n\in \Z: |S_1(n)| > 100Cg_k\slash d_k\}| \le \frac{d_k}{100}.
\end{equation} 

Therefore for $n\in J_k$ and outside of this exceptional set we get the desired estimate. Since $|J_k| - \frac{1}{100}d_k = \frac{99}{100}d_k > \frac{d_k}{2}$ the lemma is proved.
}


\end{proof}

\begin{lemma}
\label{lemm3}
Let $N_k$ be the sets from Lemma \ref{lemm2}. Then there exists $\vep>0$ such that for any k and for any 
$n\in N_k$ the zero $t$ of the function in the interval $(n, n+1)$ 
belongs to $(n+\vep, n+1-\vep)$.
\end{lemma}

\begin{proof}
Assume that $M(n+\frac{1}{2}) >0$. We have 
$$
M'(t) = \summ_{m\in \Z} \frac{|G(m)|^2}{(t - m)} \ge \frac{|G(n)|^2}{(t - n)^2}.
$$ 
Since $|G(n)|^2 \ge C \frac{g_k}{d_k}$ and $M\big(n+\frac{1}{2}\big) \le C\frac{g_k}{d_k}$ 
there exists $\eps>0$ (depending on $C$ but not on $k$ and $n$)
such that $M(n+\eps) \le M\big(n+\frac{1}{2}\big) - |G(n)|^2\int_{n+\vep}^{n+\frac{1}{2}}\frac{dt}{(t-n)^2} < 0$. 
Thus, for some $t\in (n + \eps, n + \frac{1}{2})$ we have $M(t) = 0$. 

In the case when $M(n+\frac{1}{2})<0$, one shows by the same argument that the root of $M$ will lie
in $(n+\frac{1}{2}, n + 1 - \eps)$. 
\end{proof}

\begin{proof}[Proof of Theorem \ref{main2}]
Assume that the system of reproducing kernels $\{K_\lambda\}_{\lambda\in \Lambda}$
is not hereditarily complete. Then there exists a nonzero function $f\in PW_\pi$, 
$$
f(z) = \sin \pi z \summ_{n\in \Z} \frac{(-1)^n \bar a_n}{z-n},
$$
and a set $\Lambda_1\subset \Lambda$ such that $f$ is orthogonal to all 
$G_\lambda, \lambda \in \Lambda_1$, and $K_\lambda$,  $\lambda \in \Lambda_2 = \Lambda \backslash \Lambda_1$. 
Our first step is to prove the following equality:
\begin{equation}
\label{below}
f(z) \summ_n \frac{a_nG(n)}{z-n} = G(z)\summ \frac{|a_n|^2}{z-n}.
\end{equation}
Indeed, note that $f(n) = (-1)^n \bar a_n$ and so  the residues at $\Z$ coincide. Since the left-hand side 
of \eqref{below}
vanishes at $\lambda\in \Lambda$, there is an entire function $T$ such that 
\begin{equation}
\label{hjk}
f(z) \summ_n \frac{a_nG(n)}{z-n} - G(z)\summ \frac{|a_n|^2}{z-n} = G(z)T(z). 
\end{equation}
It is clear that $T$ is of zero exponential type. Recall that $\frac{G}{z-\lambda} \in PW_\pi$
for any zero $\lambda$ of $G$. Hence, the left-hand side of 
\eqref{hjk} is in the class $PW_\pi +zPW_\pi$. If $T$ has at least one zero $\zeta$, 
we conclude that $G\cdot \frac{T}{z-\zeta} \in PW_\pi$, a contradiction to the fact
that $G$ is the generating function of a complete sequence of reproducing kernels. 
Thus, $T = c\in \co$ and
$$ 
f(z) \summ_n \frac{a_nG(n)}{z-n} = G(z)\Big( c+ \summ \frac{a_n^2}{z-n}\Big).
$$

It remains to exclude the case when $c\ne 0$.
Put $E_k = J_k \setminus \cup_{n\in\Z} (n-1/10, n+1/10)$. Then, for $x\in E_k$,  
$$
\bigg|\sum_n \frac{a_n G(n)}{x-n} \bigg|^2 \le
\sum_{n\notin I_k} \frac{|G(n)|^2}{(x-n)^2}\sum_{n\notin I_k} |a_n|^2  
+ \sum_{n\in I_k}|G(n)|^2\sum_{n\in I_k} \frac{|a_n|^2}{(x-n)^2} \lesssim g_k
$$
by (i). Note also that 
$\summ_n \frac{a_n^2}{x-n} \to 0$ when $x\to \infty$, $x\in E_k$. Thus,
$$
\int_{E_k} 
\bigg| f(x) \summ_n \frac{a_nG(n)}{x-n} \bigg| ^2 dx \lesssim g_k \int_{E_k} |f(x)|^2 dx = o(g_k),
$$ 
as $k\to\infty$, while, if $c\ne 0$, 
$$
\int_{E_k}  \bigg|G(x)\Big( c+ \summ \frac{a_n^2}{x-n}\Big)\bigg|^2 dx
\gtrsim \int_{E_k} |G(x)|^2 dx \gtrsim \frac{g_k}{d_k} \cdot d_k =  g_k
$$
by (ii). This contradiction shows that $c=0$ and \eqref{below} is proved.

Now, put 
$$
H(z) = \summ_n \frac{a_n G(n)}{z-n}, \qquad
h(z) = \summ_n \frac{a_n^2}{z-n}.
$$
Since $H\in \Hi$ and $\{\mathcal{K}_{t_n}\}$ is an orthogonal system in $\Hi$, we have  
$\Big\{\frac{H(t_n)}{\|\mathcal{K}_{t_n}\|_{\Hi}} \Big\} \in\ell^2$.
Also, $\{f(t_n)\}\in \ell^2$ by the classical Plancherel--P\'olya inequality. 
Equality \eqref{below} yields
$$
h(t_n)\frac{G(t_n)}{\|\mathcal{K}_{t_n}\|_{\Hi}}
= \frac{H(t_n)}{\|\mathcal{K}_{t_n}\|_{\Hi}} f(t_n) \in\ell^1.
$$

Denote by $\{\tilde t_j\}$ the sequence of all zeros of $M$ which belong to 
the intervals $(n, n+1)$ for $n\in N_k$, where the sets $N_k$ are constructed in Lemma \ref{lemm2}. 
By Lemma \ref{lemm3} we have ${\rm dist} (\{\tilde t_j\}, \Z) \ge \vep>0$. Hence, 
if $\tilde t_j \in (n, n+1)$, then 
$\|\mathcal{K}_{\tilde t_j}\|_\Hi^2/\|\mathcal{K}_{n + \frac{1}{2}}\|_\Hi^2$ is 
bounded from above and from below by some positive constants depending on $\vep$, 
since this is true for all the summands in their definitions (see \eqref{rep}).
By Lemma \ref{lemm2} and (iii) we have 
$$
|G(\tilde t_j)| \ge C \sqrt{\frac{g_k}{d_k}},
\qquad 
\|\mathcal{K}_{\tilde t_j}\|_\Hi \le C \sqrt{\frac{g_k}{d_k}}
$$
for any $\tilde t_j \in J_k$. We conclude that $h(\tilde t_j) \in \ell^1$. Also, 
since $|N_k| \ge d_k/2$ we have
$$
\summ_j \frac{1}{\tilde t_j}  = \summ_k \summ_{\tilde t_j \in J_k} \frac{1}{\tilde t_j}\asymp 
\summ_k \frac{|N_k|}{\rho_k} \asymp \summ_k \frac{d_k}{\rho_k} =\infty.
$$
Applying Lemma \ref{lem1} to $h(z) = \sum_{n\in \Z}\frac{|a_n|^2}{z-n}$ 
(note that we do not exclude the case that some of $a_n$ are zero) 
and $\tilde t_j$ in place of $\gamma_n$ we conclude that 
$a_n \equiv 0$. This contradiction proves the theorem.
\end{proof}

\begin{remark}
{\rm Formally, Theorems \ref{main1} and \ref{main2} apply to different classes of functions $G$, since
condition (iii) of Theorem \ref{main2} is inconsistent with the existence of any roots of $G$ on $I_k$
(condition (c) in Theorem \ref{main1}). However, in Theorem \ref{main1} we may assume that 
$\Lambda \subset (\R + i) \cup (\R - i)$ and $\Lambda = \bar \Lambda$ keeping all other  
conditions. Thus there exists a class of functions $G$ for which hereditary completeness  
depends only on finiteness of the logarithmic length of the intervals $I_k$. }
\end{remark}
{
\begin{remark}
{\rm We can allow some functions $G$ with divergent sum $\sum_{n\in \Z} \frac{G^2(n)}{n}$ with the following modification of the above method: we consider the function $M(x) = \sum_{n\in \Z} G^2(n)\left (\frac{1}{x-n} + \frac{n}{n^2+1}\right)$ and on the interval $I_k$ we will consider the points $\lambda_n\in (n, n+1)$ which are the solution to the equation 
\begin{equation}
M(x) = \summ_{n \le \pho_k + d_k} \frac{nG^2(n)}{n^2+1}.
\end{equation}

Although normalized reproducing kernels are not an orthonormal sequence anymore they come in big groups of pairwise orthogonal kernels corresponding to one interval $I_k$.  Thus one can still prove that they form a Riesz sequence by examining the Gram matrix and Riesz sequence is sufficient for our proof. This in particular allows us to consider $|G|\asymp 1$ on $I_k$ and (with slight modifications) even non-lacunary case $|G|\asymp 1$ on $\R$.   }
\end{remark}
}
\bigskip

\section{A remark on exponential parts of incomplete mixed system}
\label{s4}

Assume that $\{e^{i\lambda t}\}_{\lambda\in \Lambda}$ is a nonhereditary complete system 
of exponentials and the system 
$\{K_{\lambda}\}_{\lambda\in \Lambda_1} \cup \big\{\frac{G(z)}{z-\lambda} \big\}_{\lambda\in \Lambda_2}$
is incomplete for some partition $\Lambda = \Lambda_1\cup \Lambda_2$. We have seen in Section \ref{exa}
that $\Lambda$ can be a sufficiently small perturbation of integers. However, the system must also have 
a certain irregularity. We will show that the exponential part $\{e^{i\lambda t}\}_{\lambda\in \Lambda_1}$
cannot be a part of a Riesz basis of exponentials with some additional properties. E.g., 
$\Lambda_1$ cannot be a subset of $\Z$.

\begin{theorem}
\label{main3}
Let $\{e^{i\lambda t}\}_{\lambda\in \Lambda}$  
be a complete and minimal system in $L^2(-\pi, \pi)$, $\Lambda\subset \R$.
Assume that $\Lambda_1 \subset \Lambda$ and there exists $\tilde \Lambda_2 \subset \R$ such that
$\{e^{i\lambda t}\}_{\lambda\in \Lambda_1 \cup \tilde \Lambda_2}$ is a Riesz basis in $L^2(-\pi, \pi)$
whose generating function $F$ satisfies $|F'(\zeta)|\lesssim 1$, $\zeta\in \mathcal{Z}_F$. 
Then the system 
$$
\{e^{i\lambda t}\}_{\lambda\in \Lambda_1} \cup \{w_\lambda\}_{\lambda\in \Lambda_2},
$$
where $\Lambda_2 = \Lambda\setminus\Lambda_1$
and $\{w_\lambda\}$ is the system biorthogonal to $\{e^{i\lambda t}\}_{\lambda\in \Lambda}$,
is complete in $L^2(-\pi, \pi)$. 
\end{theorem}

\begin{proof}
Since Riesz bases of exponentials are stable under small perturbations (even in Euclidean metric)
we can perturb slightly $\tilde \Lambda_2$ so that $\Lambda_2\cap \tilde \Lambda_2 = \emptyset$ 
and still $|F'(\zeta)|\lesssim 1$, $\zeta\in \mathcal{Z}_F$. We also assume in what follows that 
$\Lambda_2$ and $\tilde \Lambda_2$ are infinite (otherwise it is well known that the 
corresponding mixed system is complete).

We pass again to the equivalent problem for the Paley--Wiener space $PW_\pi$.
Since $\{K_\zeta\}_{\zeta\in \ZZ}$, $\ZZ = \ZZ_F$, is a Riesz basis, its biorthogonal system 
$\big\{\frac{F(z)}{F'(\zeta)(z-\zeta)}\big\}_{\zeta\in \ZZ}$ also is a Riesz basis.
Consider the Hilbert space
$$
\Hi = \Big\{ f(z) = \sum_{\zeta\in \ZZ} c_\zeta \frac{F(z)}{F'(\zeta)(z-\zeta)}: (c_\zeta)\in \ell^2\Big\}
$$
with the norm $\|f\|_{\Hi} = \|(c_\zeta)\|_{\ell^2}$. Then $\Hi$ coincides with $PW_\pi$ 
with equivalence of norms and the system $\big\{\frac{F(z)}{F'(\zeta)(z-\zeta)}\big\}_{\zeta\in \ZZ}$ 
is an orthonormal basis of reproducing kernels in $\Hi$ (note that $c_\zeta = f(\zeta)$,  $f\in \Hi$).

Assume now that the system
$$
\{K_\lambda\}_{\lambda\in \Lambda_1} \cup \{G_\lambda\}_{\lambda\in \Lambda_2},\qquad
G_\lambda(z) = \frac{G(z)}{G'(\lambda)(z-\lambda)},
$$
is not complete in $PW_\pi$. This means that there exists a nonzero function $f\in PW_\pi$
such that $f|_{\Lambda_1} = 0$ and
$$
f\notin  \overline{{\rm Span}}_{PW_\pi}\big\{ G_\lambda: \ \lambda\in \Lambda_2  \big\} = 
\overline{{\rm Span}}_{\Hi}\big\{ G_\lambda: \ \lambda\in \Lambda_2  \big\}. 
$$ 
Note that the system $\{G_\lambda\}_{\lambda\in \Lambda}$ is biorthogonal 
also to the system of reproducing kernels
$\{\tilde K_\lambda\}_{\lambda\in \Lambda}$ of $\Hi$. Thus, there exists a function $f\in \Hi=PW_\pi$
such that 
$$
f\perp \{\tilde K_\lambda\}_{\lambda\in \Lambda_1} \cup \{G_\lambda\}_{\lambda\in \Lambda_2}
$$
with respect to the inner product of $\Hi$. Let 
$f(z) = \sum_{\zeta\in \ZZ} c_\zeta \frac{F(z)}{F'(\zeta)(z-\zeta)}$. Recall that $\ZZ = \Lambda_1\cup\tilde 
\Lambda_2$ whence $c_\zeta = f(\zeta) = 0$, $\zeta\in\Lambda_1$. Also,
$$
\big\langle G_\Lambda, f\big\rangle_{\Hi} = 
\frac{1}{G'(\lambda)} \sum_{\zeta\in \tilde \Lambda_2} \frac{G(\zeta) \overline c_\zeta}{\zeta-\lambda} =0, \qquad \lambda_\in \Lambda_2,
$$
and so 
$$
F(z) \sum_{\zeta\in \ZZ} \frac{G(\zeta) \overline c_\zeta}{z-\zeta} = G(z)T(z)
$$
for some entire function $T$ of zero exponential type. Since
$\big\{\frac{G(\zeta)}{|\zeta|+1}\big\} \in \ell^2$, the left-hand side of the above equality 
belongs to $PW_\pi +zPW_\pi$. If $T$ has at least one zero, say $z_0$, then $G\frac{T}{z-z_0} \in PW_\pi$,
a contradiction to the fact that $G$ is the generating function of a complete sequence. Thus, $T=c$ for some 
$c\in \mathbb{C}$. Comparing the values at $\zeta\in \tilde \Lambda_2$ we get
$F'(\zeta) G(\zeta) \overline c_\zeta = cG(\zeta)$ and so $|c_\zeta| = |F'(\zeta)|^{-1} |c| \gtrsim 1$, 
$\zeta\in \tilde \Lambda_2$, a contradiction.                                                   
\end{proof}

\begin{question} Is the condition $|F'(\zeta)|\lesssim 1$, $\zeta\in \mathcal{Z}_F$, essential?
If it is not satisfied, then it is possible that $\{|F'(\zeta)|^{-1}\}_{\zeta \in \tilde \Lambda_2}\in \ell^2$
when $\tilde \Lambda_2$ is sufficiently sparse. Writing $G=G_1G_2$, $F=G_1F_2$ and
$\overline c_\zeta = 1/F'(\zeta)$  we get
$$
\sum_{\zeta\in \tilde\Lambda_2} \frac{G_2(\zeta)}{F_2'(\zeta)(z-\zeta)} = \frac{G_2(z)}{F_2(z)}.
$$
This equation is possible for some $G_2$ which are ``smaller'' then $S_2$ but is it compatible
with the condition that $G$ is the generating function for a complete and minimal system?
\end{question}

\end{document}